\documentclass[reqno,12pt]{amsart}
\usepackage{amssymb}
\usepackage{amsfonts}
\usepackage{amsmath}
\usepackage{graphicx}
\usepackage{amscd,amsthm}

\newtheorem{theorem}{Theorem}
\theoremstyle{plain}

\newtheorem{corollary}{Corollary}

\newtheorem{lemma}{Lemma}

\newtheorem{remark}{Remark}

\numberwithin{equation}{section}

\input{tcilatex}

\begin{document}
\author{}
\title{}
\maketitle

\begin{center}
\thispagestyle{empty} \pagestyle{myheadings} 
\markboth{\bf Rahime Dere and
Yilmaz Simsek}{\bf Bernoulli type polynomials on Umbral Algebra}

\textbf{\Large Bernoulli type polynomials on Umbral Algebra}

\bigskip

\textbf{Rahime Dere} \textbf{and Yilmaz Simsek}\\[0pt]

Department of Mathematics, Faculty of Science University of Akdeniz TR-07058
Antalya, Turkey \\[0pt]

E-mail\textbf{: rahimedere@gmail.com and ysimsek@akdeniz.edu.tr}\\[0pt]

\textbf{{\large {Abstract}}}\medskip
\end{center}

\begin{quotation}
The aim of this paper is to investigate generating functions for
modification of the Milne-Thomson's polynomials, which are related to\ the
Bernoulli polynomials and the Hermite polynomials. By applying the Umbral
algebra to these generating functions, we provide to deriving identities for
these polynomials.
\end{quotation}

\bigskip



\section{Introduction}

Throughout of this paper, we use the following notations:%
\begin{equation*}
\mathbb{N}:=\left\{ 1,2,3,\cdots \right\} \text{ and }\mathbb{N}_{0}:=%
\mathbb{N\cup }\left\{ 0\right\} ,
\end{equation*}%
\begin{equation*}
\delta _{n,k}=\left\{ 
\begin{array}{c}
0\text{ if }n\neq k \\ 
1\text{ if }n=k,%
\end{array}%
\right.
\end{equation*}%
and

\begin{equation*}
\left( x\right) _{b}=x\left( x-1\right) ...\left( x-b+1\right) ,
\end{equation*}%
where $b\in \mathbb{N}$.

Here, we use the notations and definitions which are related to the umbral
algebra and calculus cf.\ \cite{Roman}.

Let $P$ be the algebra of polynomials in the single variable $x$ over the
field complex numbers. Let $P^{\ast }$ be the vector space of all linear
functionals on $P$. Let $\left\langle L\mid p(x)\right\rangle $ be the
action of a linear functional $L$ on a polynomial $p(x)$. Let $\tciFourier $
denotes the algebra of formal power series%
\begin{equation*}
f\left( t\right) =\dsum\limits_{k=0}^{\infty }\frac{a_{_{k}}}{k!}t^{k},
\end{equation*}%
cf.\ \cite{Roman}.

This kind of algebra is called an umbral algebra. Each $f\in \tciFourier $
defines a linear functional on $P$ and for all $k\geqslant 0$, $%
a_{k}=\left\langle f\left( t\right) \mid x^{k}\right\rangle $. The order $%
o\left( f\left( t\right) \right) $ of a power series $f\left( t\right) $ is
the smallest integer $k$ for which the coefficient of $t^{k}$ does not
vanish. A series $f\left( t\right) $\ for which $o\left( f\left( t\right)
\right) =1$ is called a delta series. And a series $f\left( t\right) $\ for
which $o\left( f\left( t\right) \right) =0$ is called a invertible series
cf.\ \cite{Roman}.

Let $f(t)$, $g(t)$ be in $\tciFourier $, we have%
\begin{equation}
\left\langle f(t)g(t)\mid p\left( x\right) \right\rangle =\left\langle
f(t)\mid g(t)p\left( x\right) \right\rangle ,  \label{a1}
\end{equation}%
cf.\ \cite{Roman}. For all $p\left( x\right) $ in $P$, we have%
\begin{equation}
\left\langle e^{yt}\mid p\left( x\right) \right\rangle =p\left( y\right)
\label{a2}
\end{equation}%
and%
\begin{equation}
e^{yt}p\left( x\right) =p\left( x+y\right)  \label{a21}
\end{equation}%
cf.\ \cite{Roman}.

\begin{theorem}
(\cite[p. 20, Theorem 2.3.6]{Roman}) Let $f\left( t\right) $ be a delta
series and let $g\left( t\right) $ be an invertible series. Then there exist
a unique sequence $s_{n}\left( x\right) $ of polynomials satisfying the
orthogonality conditions%
\begin{equation}
\left\langle g(t)f(t)^{k}\mid s_{n}(x)\right\rangle =n!\delta _{n,k}
\label{a3}
\end{equation}%
for all $n,k\geq 0$.
\end{theorem}

The sequence $s_{n}(x)$\ in (\ref{a3}) is the Sheffer polynomials for pair $%
(g(t),f(t))$, where $g(t)$ must be invertible and $f(t)$\ must be delta
series. The Sheffer polynomials for pair $(g(t),t)$ is the Appell
polynomials or Appell sequences for $g(t)$.

The Appell polynomials are defined by means of the following generating
function:%
\begin{equation}
\dsum\limits_{k=0}^{\infty }\frac{s_{k}\left( x\right) }{k!}t^{k}=\frac{1}{%
g(t)}e^{xt},  \label{a4}
\end{equation}%
cf.\ \cite{Roman}.

The Appell polynomials satisfy the following relations:%
\begin{equation}
s_{n}\left( x\right) =g(t)^{-1}x^{n},  \label{a5}
\end{equation}%
derivative formula%
\begin{equation}
ts_{n}\left( x\right) =s_{n}^{^{\prime }}\left( x\right) =ns_{n-1}\left(
x\right)   \label{a6}
\end{equation}%
and%
\begin{equation}
\frac{1}{t}s_{n}\left( x\right) =\frac{1}{n+1}s_{n+1}\left( x\right) ,
\label{a7}
\end{equation}%
recurrence formula%
\begin{equation}
s_{n+1}\left( x\right) =\left( x-\frac{g^{^{\prime }}\left( t\right) }{%
g\left( t\right) }\right) s_{n}\left( x\right) ,  \label{a8}
\end{equation}%
and multiplication formula, for $\alpha \neq 0$%
\begin{equation*}
s_{n}\left( \alpha x\right) =\alpha ^{n}\frac{g\left( t\right) }{g\left( 
\frac{t}{\alpha }\right) }s_{n}\left( x\right) .
\end{equation*}%
(see, for details, \cite{Roman}; and see also \cite{blasiack}, \cite{Dattoli}%
, \cite{DereSimsek}).

The remainder of this paper is organized as follows: We modify generating
functions for the Milne-Thomson's polynomials $\Phi _{n}^{(a)}(x)$. We give
some properties of this functions. By applying the Umbral algebra and Umbral
calculus, we derive some identities related to Hermite polynomials,
Bernoulli polynomials and Stirling numbers of second kind.

\section{New type polynomials}

We modify the Milne-Thomson's polynomials $\Phi _{n}^{(a)}(x)$ (see for
detail \cite{Milne}) as $\Phi _{n}^{(a)}(x,v)$ of degree $n$ and order $a$
by the means of the following generating function:%
\begin{equation}
g_{1}(t,x;a,v)=f\left( t,a\right) e^{xt+h\left( t,v\right)
}=\sum_{n=0}^{\infty }\Phi _{n}^{(a)}(x,v)\frac{t^{n}}{n!}  \label{r1}
\end{equation}%
where $f\left( t,a\right) $\ is a function of $t$ and the integer $a$.

Observe that $\Phi _{n}^{(a)}(x,0)=\Phi _{n}^{(a)}(x)$ cf. \cite{Milne}.

\begin{remark}
Setting $f\left( t,a\right) =\left( \frac{t}{e^{t}-1}\right) ^{a}$ in (\ref%
{r1}), we obtain the following polynomials by%
\begin{equation}
g_{2}(t,x;a,v)=\left( \frac{t}{e^{t}-1}\right) ^{a}e^{xt+h\left( t,v\right)
}=\sum_{n=0}^{\infty }\beta _{n}^{(a)}(x;v)\frac{t^{n}}{n!}.  \label{a9}
\end{equation}%
Observe that the polynomials $\beta _{n}^{(a)}(x;v)$ are related to not only
Bernoulli polynomials but also the Hermite polynomials. For example, if $%
h\left( t,0\right) =0$ in (\ref{a9}), we have%
\begin{equation*}
\beta _{n}^{(a)}(x,0)=B_{n}^{\left( a\right) }\left( x\right) ,
\end{equation*}%
where $B_{n}^{\left( a\right) }\left( x\right) $\ denotes the Bernoulli
polynomials of higher order which is, defined by means of the following
generating function%
\begin{equation*}
f_{B}(t,x;a)=\left( \frac{t}{e^{t}-1}\right) ^{a}e^{xt}=\sum_{n=0}^{\infty
}B_{n}^{\left( a\right) }\left( x\right) \frac{t^{n}}{n!}.
\end{equation*}%
One can easily see that $B_{n}^{\left( a\right) }\left( 0\right)
=B_{n}^{\left( a\right) }$, that is%
\begin{equation*}
f_{B}(t;a)=\left( \frac{t}{e^{t}-1}\right) ^{a}=\sum_{n=0}^{\infty
}B_{n}^{\left( a\right) }\frac{t^{n}}{n!}.
\end{equation*}%
If we take $h\left( t\right) =-\frac{vt^{2}}{2}$ in (\ref{a9}), we have%
\begin{equation*}
\left( \frac{t}{e^{t}-1}\right) ^{a}e^{xt-\frac{vt^{2}}{2}%
}=\sum_{n=0}^{\infty }\left( _{H}\beta _{n}^{(a)}(x,v)\right) \frac{t^{n}}{n!%
}.
\end{equation*}%
Hence, we get%
\begin{equation*}
_{H}\beta _{n}^{(0)}(x,v)=H_{n}^{\left( v\right) }\left( x\right)
\end{equation*}%
where $H_{n}^{\left( v\right) }\left( x\right) $\ denotes the Hermite
polynomials of higher order, which is defined by means of the following
generating function:%
\begin{equation*}
f_{H}(x,t;v)=e^{xt-\frac{vt^{2}}{2}}=\sum_{n=0}^{\infty }H_{n}^{\left(
v\right) }\left( x\right) \frac{t^{n}}{n!}.
\end{equation*}
\end{remark}

We define the following functional equation:%
\begin{equation}
g_{2}(t,x;a,v)=f_{B}(t,x;a)e^{h\left( t,v\right) }.  \label{r4}
\end{equation}%
By the above functional equation, we get%
\begin{equation}
g_{2}(t,x;a,v)=\sum_{n=0}^{\infty }B_{n}^{\left( a\right) }\left( x\right) 
\frac{t^{n}}{n!}\sum_{n=0}^{\infty }\frac{h(t,v)^{n}}{n!}.  \label{r2}
\end{equation}%
If we set $h(t,v)=-vt$ in (\ref{r2}), we have%
\begin{equation*}
\beta _{n}^{(a)}(x,v)=\sum_{j=0}^{n}(-1)^{n-j}\binom{n}{j}B_{j}^{\left(
a\right) }\left( x\right) v^{n-j}.
\end{equation*}%
We define the following functional equation:%
\begin{equation}
g_{2}(t,x;a,v)=f_{B}(t;a)e^{xt+h\left( t,v\right) }  \label{r3}
\end{equation}%
If we set $h(t,v)=-\frac{v}{2}t^{2}$ in (\ref{r4}), we obtain the following
theorem:

\begin{theorem}
\begin{equation*}
\beta _{n}^{(a)}(x,v)=\sum_{j=0}^{n}\binom{n}{j}B_{j}^{\left( a\right)
}\left( x\right) H_{n-j}^{(v)}.
\end{equation*}
\end{theorem}

From (\ref{r3}), we get%
\begin{equation*}
\frac{\partial }{\partial x}g_{2}(t,x;a,v)=tg_{2}(t,x;a,v).
\end{equation*}%
By using the above partial derivative equation, we obtain the following
theorem:

\begin{theorem}
\begin{equation*}
\frac{\partial }{\partial x}\beta _{n}^{(a)}(x,v)=n\beta _{n-1}^{(a)}(x,v).
\end{equation*}
\end{theorem}

By using (\ref{a6}) and the above theorem, it is easily to see that $\beta
_{n}^{(a)}(x,v)$ is an Appell-type sequence.

\section{Some identities for the polynomials $_{H}\protect\beta %
_{n}^{(a)}(x,v)$}

In this section, by applying the Umbral algebra and Umbral calculus, we
derive some identities related to the polynomials $_{H}\beta _{n}^{(a)}(x,v)$%
.

By substituting%
\begin{equation}
g\left( t\right) =\left( \frac{e^{t}-1}{t}\right) ^{a}e^{\frac{vt^{2}}{2}}
\label{b1}
\end{equation}%
into (\ref{a5}), one can easily obtain the following lemma:

\begin{lemma}
\label{Lemma1}Let $n\in 
\mathbb{N}
_{0}$. The following relationship holds true:%
\begin{equation*}
_{H}\beta _{n}^{(a)}(x,v)=\left( \frac{t}{e^{t}-1}\right) ^{a}e^{-\frac{%
vt^{2}}{2}}x^{n}.
\end{equation*}
\end{lemma}

By using (\ref{a6}) and (\ref{a7}), we arrive at the following lemma:

\begin{lemma}
\begin{equation}
t_{H}\beta _{n}^{(a)}(x,v)=n_{H}\beta _{n-1}^{(a)}(x,v),  \label{b2}
\end{equation}%
and%
\begin{equation}
\frac{1}{t}_{H}\beta _{n}^{(a)}(x,v)=\frac{1}{n+1}_{H}\beta
_{n+1}^{(a)}(x,v).  \label{b3}
\end{equation}
\end{lemma}

The action of a linear operator $\left( e^{t}-1\right) $ on the polynomial $%
_{H}\beta _{n}^{(a)}(x,v)$ is given by the following lemma:

\begin{lemma}
\label{Lemma2}%
\begin{equation*}
\left( e^{t}-1\right) _{H}\beta _{n}^{(a)}(x,v)=n_{H}\beta
_{n-1}^{(a-1)}(x,v).
\end{equation*}
\end{lemma}

\begin{proof}
By using Lemma \ref{Lemma1}, we obtain%
\begin{equation*}
\left( e^{t}-1\right) _{H}\beta _{n}^{(a)}(x,v)=\left( e^{t}-1\right) \left( 
\frac{t}{e^{t}-1}\right) ^{a}e^{-\frac{vt^{2}}{2}}x^{n}.
\end{equation*}%
After some calculations in the above equation, we get%
\begin{equation*}
\left( e^{t}-1\right) _{H}\beta _{n}^{(a)}(x,v)=t_{H}\beta _{n}^{(a-1)}(x,v).
\end{equation*}%
Using (\ref{b2}) in the above equation, we arrive at the desired result.
\end{proof}

From Lemma \ref{Lemma2}, we arrive at the following result:

\begin{corollary}
\begin{equation}
e_{H}^{t}\beta _{n}^{(a)}(x,v)=n_{H}\beta _{n-1}^{(a-1)}(x,v)+_{H}\beta
_{n}^{(a)}(x,v).  \label{b4}
\end{equation}
\end{corollary}

\begin{theorem}
\begin{equation*}
_{H}\beta _{n}^{(a)}(x+1,v)=n_{H}\beta _{n-1}^{(a-1)}(x,v)+_{H}\beta
_{n}^{(a)}(x,v).
\end{equation*}
\end{theorem}

\begin{proof}
Using (\ref{a21}), we get%
\begin{equation*}
e_{H}^{t}\beta _{n}^{(a)}(x,v)=_{H}\beta _{n}^{(a)}(x+1,v).
\end{equation*}%
Combining the above equation with (\ref{b4}), we complete the proof.
\end{proof}

By applying $\frac{1}{e^{t}-1}$ to the polynomial $_{H}\beta _{n}^{(a)}(x,v)$%
, we give the following lemma 

\begin{lemma}
\label{Lemma3}%
\begin{equation*}
\frac{1}{e^{t}-1}_{H}\beta _{n}^{(a)}(x,v)=\frac{1}{n+1}_{H}\beta
_{n+1}^{(a+1)}(x,v).
\end{equation*}
\end{lemma}

\begin{proof}
From Lemma \ref{Lemma1}, we get%
\begin{equation*}
\frac{1}{e^{t}-1}_{H}\beta _{n}^{(a)}(x,v)=\frac{1}{e^{t}-1}\left( \frac{t}{%
e^{t}-1}\right) ^{a}e^{-\frac{vt^{2}}{2}}x^{n}.
\end{equation*}%
After some calculations, we obtain%
\begin{equation*}
\frac{1}{e^{t}-1}_{H}\beta _{n}^{(a)}(x,v)=\frac{1}{t}_{H}\beta
_{n}^{(a+1)}(x,v).
\end{equation*}%
By using (\ref{b3}) in the above equation, we arrive at the desired result.
\end{proof}

\begin{theorem}[Recurrence formula]
\begin{equation*}
_{H}\beta _{n+1}^{(a)}(x,v)=\frac{1}{n-a+1}\left( \left( x-a\right) \left(
n+1\right) _{H}\beta _{n}^{(a)}(x,v)-a_{H}\beta _{n+1}^{(a+1)}(x,v)-n\left(
n+1\right) _{H}\beta _{n-1}^{(a)}(x,v)\right) .
\end{equation*}
\end{theorem}

\begin{proof}
By using (\ref{b1}) into (\ref{a8}), we obtain%
\begin{equation*}
_{H}\beta _{n+1}^{(a)}(x,v)=\left( x-\frac{ae^{t}}{e^{t}-1}+\frac{a}{t}%
-t\right) \left( _{H}\beta _{n}^{(a)}(x,v)\right) .
\end{equation*}

After elementary manipulations in this equation by using (\ref{b2}), (\ref%
{b3}), (\ref{b4}) and Lemma \ref{Lemma3}, we arrive at the last result.
\end{proof}

\begin{theorem}
Let $k,a\in 
\mathbb{N}
$ and $k>a$. We have%
\begin{equation*}
\left\langle \left( e^{t}-1\right) ^{k}\mid \left( _{H}\beta
_{n}^{(a)}(x,v)\right) \right\rangle =\sum_{m=0}^{\infty }\frac{\left(
-v\right) ^{2m}\left( k-a\right) !\left( n\right) _{2m+a}}{\left( m!\right)
2^{m}}S\left( n-2m-a,k-a\right) ,
\end{equation*}%
where $S\left( n-2m-a,k-a\right) $ denotes the Stirling numbers of second
kind.
\end{theorem}

\begin{proof}
Using Lemma \ref{Lemma1}, we get%
\begin{equation*}
\left\langle \left( e^{t}-1\right) ^{k}\mid \left( _{H}\beta
_{n}^{(a)}(x,v)\right) \right\rangle =\left\langle \left( e^{t}-1\right)
^{k}\mid \left( \frac{t}{e^{t}-1}\right) ^{a}e^{-\frac{vt^{2}}{2}%
}x^{n}\right\rangle .
\end{equation*}%
By\ using (\ref{a1}), we obtain%
\begin{equation*}
\left\langle \left( e^{t}-1\right) ^{k}\mid \left( _{H}\beta
_{n}^{(a)}(x,v)\right) \right\rangle =\left\langle \left( e^{t}-1\right)
^{k-a}\mid t^{a}e^{-\frac{vt^{2}}{2}}x^{n}\right\rangle .
\end{equation*}%
After some calculations, we have%
\begin{equation*}
\left\langle \left( e^{t}-1\right) ^{k}\mid \left( _{H}\beta
_{n}^{(a)}(x,v)\right) \right\rangle =\left\langle \left( e^{t}-1\right)
^{k-a}\mid \sum_{m=0}^{\infty }\frac{\left( -v\right) ^{2m}}{m!2^{m}}%
t^{2m+a}x^{n}\right\rangle .
\end{equation*}%
Thus, using (\ref{b2}) in the above equation, we get%
\begin{equation}
\left\langle \left( e^{t}-1\right) ^{k}\mid \left( _{H}\beta
_{n}^{(a)}(x,v)\right) \right\rangle =\sum_{m=0}^{\infty }\frac{\left(
-v\right) ^{2m}}{\left( m!\right) 2^{m}}\left\langle \left( e^{t}-1\right)
^{k-a}\mid \left( n\right) _{2m+a}x^{n-2m-a}\right\rangle .  \label{r0}
\end{equation}%
By substituting%
\begin{equation*}
S\left( n-2m-a,k-a\right) =\frac{1}{\left( k-a\right) !}\left\langle \left(
e^{t}-1\right) ^{k-a}\mid x^{n-2m-a}\right\rangle 
\end{equation*}%
cf.  \cite[pp. 59]{Roman} the above equation into (\ref{r0}), we obtain%
\begin{equation*}
\left\langle \left( e^{t}-1\right) ^{k}\mid \left( _{H}\beta
_{n}^{(a)}(x,v)\right) \right\rangle =\sum_{m=0}^{\infty }\frac{\left(
-v\right) ^{2m}\left( k-a\right) !\left( n\right) _{2m+a}}{\left( m!\right)
2^{m}}S\left( n-2m-a,k-a\right) .
\end{equation*}
\end{proof}

A relationship between $B_{n}^{(a)}(x)$ and$\ _{H}\beta _{n}^{(a)}(x,v)$ is
given by the following theorem:

\begin{theorem}
The following relationship holds true:%
\begin{equation*}
e^{-\frac{vt^{2}}{2}}B_{n}^{(a)}(x)=_{H}\beta _{n}^{(a)}(x,v).
\end{equation*}
\end{theorem}

\begin{proof}
Setting%
\begin{equation*}
B_{n}^{(a)}(x)=\left( \frac{t}{e^{t}-1}\right) ^{a}x^{n},
\end{equation*}%
we obtain,%
\begin{equation*}
e^{-\frac{vt^{2}}{2}}B_{n}^{(a)}(x)=e^{-\frac{vt^{2}}{2}}\left( \frac{t}{%
e^{t}-1}\right) ^{a}x^{n}.
\end{equation*}%
Using Lemma \ref{Lemma1}, we arrive at the final result.
\end{proof}

\section{Acknowledgements}

The present investigation was supported, in part, by the \textit{Scientific
Research Project Administration of Akdeniz University}.


\begin{thebibliography}{9}
\bibitem{blasiack} P. Blasiak, G. Dattoli, A. Horzela and K. A. Penson,
Representations of monomiality principle with Sheffer-type polynomials and
boson normal ordering, Phys. Lett. A 352 (2006) 7-12.

\bibitem{bretti} G. Bretti and P. E. Ricci, Multidimensional extensions of
the Bernoulli and Appell polynomials, Taiwanese Journal of Mathematics 8
(2004) 415-428.

\bibitem{Dattoli} G. Dattoli, M. Migliorati and H. M. Srivastava, Sheffer
polynomials, monomiality principle, algebraic methods and the theory of
classical polynomials, Math. Comput. Modelling 45 (2007) 1033-1041.

\bibitem{DereSimsek} R. Dere and Y. Simsek, Genocchi polynomials associated
with the Umbral algebra, In press, accepted manuscript, Appl. Math. Comput.
217 (2011), doi:10.1016/j.amc.2011.01.078.

\bibitem{Milne} L. M. Milne-Thomson, Two classes of generalized polynomials,
Proc. London Math. Soc. s2-35(1) (1933) 514-522.

\bibitem{Roman} S. Roman, The Umbral Calculus, Dover Publ. Inc. New York,
2005.
\end{thebibliography}
\end{document}